\documentclass[a4paper,10pt]{article}

\usepackage[charter]{mathdesign}

\usepackage{amsmath,amsthm,adjustbox}
\usepackage{dsfont}

\usepackage{url}

\usepackage[utf8]{inputenc}

\usepackage{xcolor}

\newtheorem{theorem}{Theorem}[section]

\newtheorem{lemma}[theorem]{Lemma}

\newtheorem{remark}[theorem]{Remark}

\newtheorem{example}[theorem]{Example}
\newtheorem{assumption}[theorem]{Assumption}

\newcommand{\N}{\mathbb{N}}
\newcommand{\Z}{\mathbb{Z}}
\newcommand{\Q}{\mathbb{Q}}
\newcommand{\R}{\mathbb{R}}

\newcommand{\B}{\mathcal{B}}

\newcommand{\F}{\mathcal{F}}

\newcommand{\X}{\mathcal{X}}

\newcommand{\U}{\mathcal{U}}

\renewcommand{\P}{\mathbb{P}}
\newcommand{\E}{\mathbb{E}}

\newcommand{\law}[1]{\mathrm{Law}(#1)}

\newcommand{\ind}{\mathds{1}}

\newcommand{\eps}{\varepsilon}
\newcommand{\la}{\lambda}
\newcommand{\ga}{\gamma}
\newcommand{\ka}{\kappa}

\newcommand{\dint}{\mathrm{d}}  

\newcommand{\leb}[1]{\mathrm{Leb}\left(#1\right)}
\newcommand{\ul}[1]{\underline{#1}}
\newcommand{\ol}[1]{\overline{#1}}

\title{Ergodic aspects of trading with threshold strategies\thanks{Both authors benefitted 
from the support of the ``Lend\"ulet'' grant LP 2015-6 of 
		the Hungarian Academy of Sciences.}}

\author{
	Attila Lovas\thanks{Alfr\'ed R\'enyi Institute of Mathematics and 
	Budapest University of Technology and Economics, Budapest, Hungary}
	\footnote{Corresponding author email: lovas.attila@renyi.hu}
	\and 
	Mikl\'os R\'asonyi\thanks{Alfr\'ed R\'enyi Institute of Mathematics and E\"otv\"os Lor\'and University, Budapest, Hungary}}

\date{\today}

\begin{document}

\maketitle

\begin{abstract} 
To profit from price oscillations, investors frequently use threshold-type strategies where changes in the portfolio 
position are triggered by some indicators reaching prescribed levels. 

In this paper we investigate threshold-type strategies in the context of ergodic control.
We make the first steps towards their optimization by proving 
ergodic properties of related functionals. Assuming Markovian price increments
satisfying a minorization condition and (one-sided) 
boundedness we show, in particular, that for given thresholds, the 
distribution of the gains converges in the long run. 

We also extend 
recent results on the stability of overshoots of random walks from the i.i.d.\ increment case to Markovian increments, 
under suitable conditions. 
\end{abstract}


\section{Introduction}\label{sec:intro}

Perhaps the most naive approach to speculative trading is trying to ``buy low and sell high'' a 
given financial asset. More refined versions of such strategies are actually widely used by practitioners, see 
\cite{www1,www2,www3}. Their various aspects have been analysed in several papers, see e.g.\ \cite{dai,zervos,zz,zhang} and
our literature review in Section 3. 

We intend to study such strategies in a different setting: that of ergodic control.
A rigorous mathematical formulation turns out to pose thorny questions about the ergodicity of certain processes, as 
we shall point out below.

The present article starts to build a reasonable and mathematically sound framework for investigating such problems. We establish that 
key functionals converge to an invariant distribution and obey a law of large numbers. We are unaware of 
any previous study that would tackle these questions. Our results may serve as a basis for further 
related investigations, using techniques of ergodic and adaptive control.

We also investigate a closely related object, studied in \cite{mijatovic-vysotsky,mijatovic2020stationary}: the so-called 
overshoot process. We extend certain results from \cite{mijatovic-vysotsky} from i.i.d.\ to Markovian summands. 

The paper is organized as follows: 
In Section \ref{sec:stability}, we state our main results on the stability of level crossings and related quantities of 
random walks with Markovian martingale differences satisfying minorization and (one-sided) boundedness. In Section 
\ref{sec:trading}, we explain the financial setting and the significance of our results in studying optimal trading 
with threshold strategies. Section \ref{sec:overshoots} presents our results about overshoots. Section 
\ref{sec:proofs} contains the proofs of the main results. 
Section \ref{sec:con} dwells upon future
directions of research.

\section{Stability of level crossings}\label{sec:stability}

Let $M>0$ and let $(X_n)_{n\in\N}$ be a time-homogeneous Markov chain on the probability space $(\Omega,\F,\P)$ with state space 
$\X:=(-\infty,M]$. Its transition kernel is denoted 
$P:\X\times \B (\X)\to [0,1]$. We consider the random walk
\begin{equation}\label{eq:rw}
	S_n = S_0 + X_1 + \ldots + X_n,
\end{equation}
where $S_0$ is a random variable independent of $\sigma (X_k:\,k\in\N)$. 

The next minorization condition ensures that the chain jumps, with positive probability, in 
one step to a small neighborhood of zero independently of the initial state. Moreover, the random movements 
of $S$ have a diffuse component.

\begin{assumption}\label{as:jmpDoeblin}
	There exist $\alpha,h>0$ such that, for all $x\in\X$ and $A\in\B (\X)$,
	\begin{equation}\label{minocond}
	P (x,A)\ge \alpha\ell(A)
	\end{equation} 
	holds where 
	\begin{equation*}
		\ell(A):=\frac{1}{2h}\leb{[-h,h]\cap A}
	\end{equation*}
	is the normalized Lebesgue measure on $[-h,h]$.
\end{assumption}

\begin{lemma}\label{lem:XlimD}
	Under Assumption \ref{as:jmpDoeblin}, there is a unique probability $\pi_{*}$ on $\B(\X)$
	such that $\law{X_{n}}\to \pi_{*}$ in total variation as $n\to\infty$, at a geometric
	speed.	
\end{lemma}
\begin{proof} Assumption \eqref{minocond} implies that the state space $\mathcal{X}$ is a small set.
	Hence the chain is uniformly ergodic by Theorem 16.2.2 of \cite{mt}. 
\end{proof}

\begin{assumption}\label{as:zeroMean}
For each $z\in\mathcal{X}$, 
	\begin{equation*}
	\int_{\X} |x| \,P(z,\dint x)<\infty
	\end{equation*}
and
	\begin{equation*}
	\int_{\X} x \,P(z,\dint x)=0,
	\end{equation*}
\end{assumption}

\begin{remark}\label{martingale} {\rm Clearly, Assumption \ref{as:zeroMean} guarantees that $X_n$, $n\in\mathbb{N}$, if integrable, is
a martingale difference sequence and the limit distribution of Lemma \ref{lem:XlimD} has zero mean, that is,
	\begin{equation*}
	\int_{\X} x \,\pi_{*}(\dint x)=0.
	\end{equation*}
}
\end{remark} 

Let us fix thresholds $\ul{\theta},\ol{\theta}\in\R$, satisfying $\ul{\theta}<0<\ol{\theta}$. Furthermore, 
we define the sequence of crossing times corresponding to $\ul{\theta}$ and $\ol{\theta}$ by 
the recursion 
$L_{0}:=0$ and for $n\in\N$,
\begin{equation}\label{eq:defLT}
T_{n+1}:=\min\{k>L_{n}:\ S_{k}<\ul{\theta}\},\quad
L_{n+1}:=\min\{k>T_{n+1}:\ S_{k}>\ol{\theta}\}.
\end{equation}

\begin{lemma}\label{lem:TLfinite} Let $EX_{n}^{2}<\infty$ hold for all $n\geq 1$. (This is the case, in particular,
if the $X_{n}$ are bounded.) 
	Under Assumptions \ref{as:jmpDoeblin} and \ref{as:zeroMean}, the random variables $T_{n},L_{n}$ 
	are well-defined and almost surely finite.
\end{lemma}
\begin{proof}

	We will prove the statement inductively, the first step being trivial since $L_{0}<\infty$.{}
	Assume that the statement has been shown for $L_{0},L_{1},T_{1},\ldots,L_{n}$ and
	we go on showing it for $T_{n+1}$ and $L_{n+1}$.

In the induction step, we work on the events $B_{k}:=\{L_{n}=k\}$, $k\in\N$ separately. Fixing $k$, 
the process $M_{j}:=\sum_{l=0}^{j}\ind_{B_{k}}X_{k+l}$, $j\in\N$ is a square-integrable martingale 
(remember Remark \ref{martingale}) with conditional quadratic variation 
	\begin{equation*}
	Q_{j}:=\E[(M_{j}-M_{j-1})^{2}\mid\sigma(M_{i},i\leq j-1)]\geq \ind_{B_{k}}\frac{\alpha}{2h}\int_{-h}^{h}y^{2}\, dy=
	\ind_{B_{k}}\frac{\alpha h^{2}}{3},
	\end{equation*}
	using Assumption \ref{as:jmpDoeblin}. Hence $\sum_{j=1}^{\infty}Q_{j}=\infty$ almost surely on $B_{k}$. Proposition VII-3-9. of \cite{neveu}
	implies that $\liminf_{j\to\infty}M_{j}=-\infty$ on $B_{k}$. It follows that, on $B_{k}$, almost
	surely 
	\begin{equation*}
	\liminf_{j\to\infty}\,(S_{0}+X_{1}+\ldots+X_{L_{n}}+\ldots+X_{L_{n}+j})=-\infty 
	\end{equation*} 
	which implies, in particular, $T_{n+1}<\infty$ on $B_{k}$.
A similar argument establishes that also $P(L_{n+1}<\infty)=1$.
\end{proof}

Although $X_{T_1}$ can be positive when $S_0<\ul{\theta}$, for $n\ge 2$, $X_{T_n}$ is always negative. 
Moreover it is also straightforward to verify that the process 
\begin{equation*}
	U_{n}:=(X_{T_{n}},S_{T_{n}},X_{L_{n}},S_{L_{n}},L_n-T_n),\ n\geq 2,
\end{equation*}
is a time-homogeneous Markov chain on the state space 
\begin{equation*}
\U:=(-\infty,0)\times (-\infty,\ul{\theta})\times (0,M]\times (\ol{\theta},\ol{\theta}+M)\times 
(\N\setminus \{0\}).
\end{equation*}

The next theorem states that under our standing assumptions, the law of $U_n$ converges to a unique limiting law, as 
$n\to\infty$, moreover, bounded functionals of $U_n$ admit an ergodic behavior.

\begin{theorem}\label{cor:LLN} Let $EX_{n}^{2}<\infty$ hold for all $n\geq 1$.
	Under Assumptions \ref{as:jmpDoeblin} and \ref{as:zeroMean}, there exists a probability $\upsilon$ on $\B(\U)$ such that $\law{U_n}\to\upsilon$ at a geometric speed in total variation as $n\to\infty$.
	Furthermore, for any bounded and measurable function $\phi:\U\to\R$,
	\begin{equation}\label{lili}
		\frac{\sum_{j=1}^{n}\phi(U_{j})}{n}\to \int_{\U}\phi(u)\,\upsilon(\dint u),\ n\to\infty,
	\end{equation}	
almost surely.
\end{theorem}
\begin{proof}	
	See in Section \ref{sec:proofs}.	
\end{proof}

A ``mirror image'' of the proof of Theorem \ref{cor:LLN} establishes the following
result, the ``symmetric pair'' of Theorem \ref{cor:LLN}.

\begin{theorem}\label{cor:LLN1} Let $\tilde{X}_{t}$, $t\in\mathbb{N}$ be a Markov chain on the state space
$[-M,\infty)$ for some $M>0$ and define $\tilde{S}_{n}:=S_{0}+\sum_{k=1}^{n}\tilde{X}_{k}$. 
Let $E\tilde{X}_{n}^{2}<\infty$ hold for all $n\geq 1$. Let Assumptions \ref{as:jmpDoeblin} and \ref{as:zeroMean} hold for $\tilde{X}_{t}$
Then the recursively defined quantities $\tilde{T}_{0}:=0$ and
\begin{equation}
\tilde{L}_{n+1}:=\min\{k>\tilde{T}_{n}:\ \tilde{S}_{k}>\ol{\theta}\},\quad
\tilde{T}_{n+1}:=\min\{k>\tilde{L}_{n+1}:\ \tilde{S}_{k}<\ul{\theta}\}
\end{equation}
are well-defined and finite. Furthermore,
there exists a probability $\upsilon$ on $\B(\tilde{\U})$ such that $\law{\tilde{U}_n}\to\upsilon$ at a geometric speed 
in total variation as $n\to\infty$, where
\begin{equation*}
\tilde{U}_{n}:=(\tilde{X}_{\tilde{L}_{n}},\tilde{S}_{\tilde{L}_{n}},\tilde{X}_{\tilde{T}_{n}},\tilde{S}_{\tilde{T}_{n}},\tilde{T}_n-
\tilde{L}_n),\ n\geq 2
\end{equation*}
is a homogeneous Markov chain on the state space
\begin{equation*}
\tilde{\mathcal{U}}:=(0,\infty)\times (\ol{\theta},\infty)\times [-M,0)\times (\ul{\theta}-M,\ul{\theta})\times 
(\N\setminus \{0\}).
\end{equation*}
For any bounded and measurable function $\phi:\tilde{\mathcal{U}}\to\R$,
	\begin{equation}\label{lili}
		\frac{\sum_{j=1}^{n}\phi(\tilde{U}_{j})}{n}\to \int_{\tilde{\mathcal{U}}}\phi(u)\,\upsilon(\dint u),\ n\to\infty,
	\end{equation}	
almost surely.\hfill $\Box$
\end{theorem}

\section{Trading with threshold strategies}\label{sec:trading}

Let the observed price of an asset be denoted by $A_{t}$ at time $t\in\N$. 
We may think of the price of a futures contract, for instance. Positive
prices can also be handled, see Remark \ref{posprice} below. 
We assume a simple dynamics:
\begin{equation}\label{eq:logprice}
	A_t=\mu t +S_t,
\end{equation} 
where $\mu\in\R$, $S_{0}\in\mathbb{R}$ are constants and $S_{t}:=S_{0}+\sum_{j=1}^{t}X_{j}$ for a Markov process $X$ 
with values in $[-M,M]$ for some $M>0$ and satisfying Assumptions \ref{as:jmpDoeblin} and \ref{as:zeroMean}. 

In this price model, $\mu t$ represents the drift (or ``trend'') and $S_t$ 
performs fluctuations around the trend (the martingale part). 
The minorization condition \eqref{minocond} is easy to interpret: whatever the current increment 
$x\in [-M,M]$ of the fluctuations 
$S$ is, with a positive probability the next increment will be small (that is, the price change
will be close to
$0$) and the movements of $S$ are diffuse, more precisely, they have a 
diffuse component. 

The practical situation we have in mind is an algorithm that tries to ``buy low and sell high''
an asset at high (but not ultra-high) frequencies, revising the portfolio, say, once every second or every minute. 
Such an algorithm is run continuously during the
trading day which can be considered a ``stationary'' environment as economic fundamentals do
not change significantly on such timescales. 

It seems that ergodic stochastic control is the right settings for such investment
problems: the algorithms does the same thing ``forever'' and its average perfomance should
be optimized. We remark that $\mu$ in such a setting is negligible and can safely be assumed $0$, as
often done in papers on high-frequency trading. Our results work nevertheless for arbitrary $\mu$ which
is of interest for trading on different timescales (e.g.\ daily revision of a portfolio for several 
months).

Now we set up the elements of our trading mechanism. Let the thresholds $\ul{\theta},\ol{\theta}\in\R$ be 
fixed, satisfying $\ul{\theta}<0<\ol{\theta}$. We interpret $\ul{\theta}$ as a level for $S_{t}$ under 
which it is advisable to buy the asset. Analogously, it is recommended to \emph{sell} the asset if $S_{t}$ 
exceeds $\ol{\theta}$. Thus, to realize a ``buying low, selling high''-type strategy, 
the asset should be bought at the times $T_{n}$, $n\geq 1$ and sold at the times $L_{n}$,
$n\geq 1$, realizing the profit 
\begin{equation}\label{eq:profit}
A_{L_{n}}-A_{T_{n}}=S_{L_n}-S_{T_n}+\mu(L_n-T_n).
\end{equation}

We now explain the significance of Theorem \ref{cor:LLN} in studying optimal trading with threshold 
strategies. An investor aims to maximize in $\underline{\theta},\overline{\theta}$
the long-term average utility from wealth, that is,
\begin{equation}\label{eq:LTutil}
	\limsup_{n\to\infty}\frac{\sum_{k=1}^n 
		\left[u(S_{L_k(\theta)}-S_{T_k(\theta)}+\mu(L_k(\theta)-T_k(\theta)))-p(L_{k}(\theta)-T_{k}(\theta))\right]}{n}, 
\end{equation}
where $u:\R\to\R$ is a utility function and $L_{k}(\theta),T_{k}(\theta)$ refer to
the respective stopping times defined in terms of the parameter
$\theta=(\underline{\theta},\overline{\theta})$. The function $p:\mathbb{R}_{+}\to \mathbb{R}_{+}$
serves to penalize long waiting times. 

\begin{remark}{\rm If the price is modelled by processes with continuous trajectories, as in
\cite{cartea,dai,zervos,zhang} then the thresholds are hit precisely and the profit realized
between $T_{n}$ and $L_{n}$ is exactly $\overline{\theta}-\underline{\theta}$. In the present
setting (just like in the case of continuous-time processes with jumps), the profit
realized may be significantly different due to the overshoot (resp.\ undershoot) of the level $\overline{\theta}$
(resp.\ $\underline{\theta}$). From the point of view of ergodic control, it is crucial to establish
that these overshoots/undershoots tend to a limiting law, which is the central concern of our present paper.
}
\end{remark}

According to Theorem \ref{cor:LLN}, the limsup in the above expression is, in fact, a limit, for a large
class of $u$, $p$. One could easily incorporate various types of transaction costs in the model.

\begin{example}{\rm 
Let $u,p$ be non-decreasing functions that are bounded from above (e.g.\ $u$ can be the exponential utility, expressing
high risk-aversion),
and assume $\mu\geq 0$. Since $S_{L_{k}}-S_{T_{k}}$ is necessarily bounded \emph{from below}, \eqref{lili} 
holds with the choice $\phi(U_{n}):=
u(S_{L_k}-S_{T_k}+\mu(L_k-T_k))-p(L_{k}-T_{k})$ and the limsup is a limit in \eqref{eq:LTutil} above.}
\end{example}

\begin{remark}\label{symfm}{\rm In the alternative setting of Theorem \ref{cor:LLN1} above (with $\tilde{X}_{n}=X_{n}$),
the trader sells one unit of the financial asset at $\tilde{L}_{n}$ (shortselling) and then closes the position
at $\tilde{T}_{n}$
thus realizing a profit 
\begin{equation}\label{eq:profit2}
-(A_{\tilde{T}_{n}}-A_{\tilde{L}_{n}})=S_{\tilde{L}_n}-S_{\tilde{T}_n}-\mu(\tilde{L}_n-\tilde{T}_n).
\end{equation}
This is the analogue (with short positions) of the long-position strategy realizing \eqref{eq:profit}.
Theorem \ref{cor:LLN1} implies that 
\begin{equation}\label{eq:LTutil}
	\limsup_{n\to\infty}\frac{\sum_{k=1}^n 
		\left[u(S_{\tilde{L}_k(\theta)}-S_{\tilde{T}_k(\theta)}+\mu(\tilde{L}_k(\theta)-\tilde{T}_k(\theta)))-{}
		p(\tilde{L}_{k}(\theta)-\tilde{T}_{k}(\theta))\right]}{n} 
\end{equation}
is a limit in this case, too.
}
\end{remark}

In future work, we intend to optimize $\ul{\theta},\ol{\theta}$ by means of adaptive control, using recursive schemes 
such as the Kiefer-Wolfowitz 
algorithm, see \cite{zh} and Section 6 of \cite{kinga}. To prove the convergence of such procedures, it is a prerequisite 
that the process ${S_{L_k}-S_{T_k}+\mu(L_k-T_k)}$, $k\in\mathbb{N}$ has favorable ergodic properties. 
This is precisely the content of 
Theorem \ref{cor:LLN} above. 

\begin{remark}\label{posprice}{\rm In an alternative setting, $A_{t}$ may model the \emph{logprice} of an asset. In that case 
investing one dollar between $T_{n}$ and $L_{n}$ yields $\exp\left(S_{L_{n}}-S_{T_{n}}+\mu(L_{n}-T_{n})\right)$ dollars. 
Let $u:(0,\infty)\to \R$, $p:\mathbb{R}_{+}\to\mathbb{R}_{+}$ be non-decreasing functions, $p$ bounded and
$u$ bounded from above 
(such as a negative power utility function). 
In this setting the optimization
\begin{equation*}
\max_{\underline{\theta},\overline{\theta}}\limsup_{n\to\infty}\frac{\sum_{k=1}^n 
\left[u(\exp\left(S_{L_k(\theta)}-S_{T_k(\theta)}+\mu(L_k(\theta)-T_k(\theta))\right))-p(L_{k}(\theta)-T_{k}(\theta))\right]}{n}
\end{equation*}
corresponds to maximizing the utility of the long-term investment of one dollar (minus an impatience
penalty), using threshold
strategies controlled by $\theta$. When $\mu\geq 0$, the limsup is a limit, again by Theorem \ref{cor:LLN}.
}
\end{remark}

We briefly compare our approach to existing ones. We do not survey the large literature on switching
problems, see Chapter 5 of \cite{pham}, only some of the directly related papers.  
Formulations as optimal stopping problems with discounting appear e.g.\ in \cite{shiryaev,dai,zhang}.
Sequential buying and selling decisions are considered for mean-reverting assets in \cite{zz} and \cite{song}. 
Mean-reversion trading is also analysed in \cite{leung}. \cite{zervos}
treats a general diffusion setting, again
using discounting. 

In our setting of intraday
trading discounting is not an appealing option: on such timescales the decrease of the value of future money is not
manifested. Here the ergodic control of averages seems more natural an objective
to us. Recall also \cite{cartea} exploring high-frequency perspectives maximizing expectation 
on a finite horizon (without discounting).

All the above mentioned papers are about diffusion models where the phenomenon  of ``overshooting'' and
``undershooting'' does not appear. They are, on the contrary, the main focus of the present work. 
Similar problems seem to come up in an ergodic control setting for
continuous-time price processes with jumps. We are unaware of any related studies.

\section{Stability of overshoots}\label{sec:overshoots}

In \cite{mijatovic-vysotsky} the authors consider a zero-mean i.i.d.\ sequence $X_{n}$, $n\geq 1$
and a random variable $S_{0}$ independent of the $X_{t}$.
They determine the (stationary) limiting law $\mu_{*}$ for the Markov process of \emph{overshoots} defined 
by $O_{0}:=\max (S_{0},0)$,
$$
O_{n}=S_{L_{n}},\ n\geq 1 
$$
where $L_{n}, T_{n}$ are defined as in \eqref{eq:defLT} but with the choice 
$\underline{\theta}=\overline{\theta}=0$\footnote{Strictly speaking,
in the definition of $L_{n+1}$, see \eqref{eq:defLT}, they have $\geq$ instead of $>$.}. 
They also establish the convergence of
$\mathrm{Law}(O_{n})$ to $\mu_{*}$ under suitable conditions.
Generalizations to \emph{entrance Markov chains} on more general state spaces have been obtained
in \cite{mijatovic2020stationary}.

Using methods of the present paper, we may obtain generalizations into another direction: we
may relax the independence assumption on the $X_{t}$. 

\begin{theorem}\label{thm:ovi} 
	Under Assumption \ref{as:jmpDoeblin} and \ref{as:zeroMean}, there exists
	a probability $\upsilon$ on $\B((0,\infty))$ such that $\law{O_n}\to\upsilon$ at a geometric 
	speed in total variation as $n\to\infty$.
	Furthermore, for any bounded and measurable function $\phi:(0,\infty)\to\R$,
	\begin{equation}
		\frac{\sum_{j=1}^{n}\phi(O_{j})}{n}\to \int_{(0,\infty)}\phi(u)\,\upsilon(\dint u),\ n\to\infty,
	\end{equation}	
almost surely.
\end{theorem}
\begin{proof}	
	See in Section \ref{sec:proofs}.	
\end{proof}

\begin{remark}{\rm In the i.i.d.\ case Theorem \ref{thm:ovi} applies if $X_{t}$ is square-integrable, bounded from above
and the law of $X_{t}$ dominates constant times the Lebesgue measure in a neighbourhood of $0$.
In \cite{mijatovic-vysotsky} a much larger class of i.i.d.\ random variables is treated. On the other hand,
we can handle Markovian summands unlike \cite{mijatovic-vysotsky}.}
\end{remark}

\section{Proofs}\label{sec:proofs}

\begin{proof}[Proof of Theorem \ref{cor:LLN}]
	
Iterated random function representation of Markov chains on standard Borel spaces is a commonly used construction, see e.g. 
\cite{bwbook,bm}. A similar representation for $(X_n)_{n\in\N}$ is shown in Lemma \ref{lem:ifs} below which will play a crucial role in the proof. Although the proof is quite standard,
we present it for the reader's convenience.
\begin{lemma}\label{lem:ifs}  
	Let $(\xi_n)_{n\in\N}$ and $(\eta_n)_{n\in\N}$ be i.i.d. sequences, independent of each other, and
	also independent of $\sigma (X_0, S_0)$, moreover let $\xi_0,\eta_0$ be uniform on $[0,1]$. Then there exists a
	map $\Phi:\X\times [0,1]\times [0,1]\to\X$ such that for all $x\in\X$, and $u\in [0,1]$,
	we have
	\begin{equation}\label{eq:prop}
	\forall v\in [0,\alpha) \quad	\Phi (x, u, v) = h (2u-1),
	\end{equation}
	where $h,\alpha>0$ are as in Assumption \ref{as:jmpDoeblin}.
	Furthermore, the process $(X_n')_{n\in\N}$ given by the recursion $X_0'= X_0$,
	$X_{n+1}' = \Phi (X_n', \xi_{n+1}, \eta_{n+1})$, $n\in\N$
	is a version of $(X_n)_{n\in\N}$.
\end{lemma}
\begin{proof}
	For $x\in\X$ and $A\in \B (\X)$, let us consider the decomposition
	\begin{equation*}
	P(x,A) = \alpha \ell(A) + (1-\alpha) q(x,A),
	\end{equation*}
	where by Assumption \ref{as:jmpDoeblin},
	\begin{equation*}
	q(x,A)=\frac{P (x,A)-\alpha\ell(A)}{1-\alpha}
	\end{equation*}
	is a probability kernel. For $x\in\X$ and $u,v\in [0,1]$, we define 
	\begin{equation*}
	\Phi (x,u,v) = \ind_{\{v<\alpha\}}\, h (2u-1) + \ind_{\{v\ge\alpha\}}\,q^{-1}(x,u),
	\end{equation*}
	where 
	$q^{-1}(x,u):=\inf\{r\in\Q \mid q(x, (-\infty,r])\ge u\}$, $u\in [0,1]$
	is the pseudoinverse of the cumulative distribution function $r\mapsto q(x,(-\infty,r))$, $x\in\X=(-\infty,M]$.
	
	Obviously, \eqref{eq:prop} holds true, and thus for any fixed $u\in [0,1]$, the random map $x\mapsto\Phi (x, u, \eta_{n+1})$ is constant on $\X$ with probability $\alpha$ showing that $X_n'$ forgets its previous state with positive probability. This observation will play a central role later.
	
	On the other hand, by the definition of the pseudoinverse, $\law{q^{-1}(x,\xi_0)}=q(x,\cdot)$,
	and thus we can write
	\begin{align*}
	\P (\Phi (x,\xi_0,\eta_0)\in A) &=
	\P (\Phi (x,\xi_0,\eta_0)\in A;\eta_0<\alpha)
	+
	\P (\Phi (x,\xi_0,\eta_0)\in A;\eta_0\ge\alpha)
	\\
	&=\alpha\P (h(2\xi_0-1)\in A) + (1-\alpha)\P (q^{-1}(x,\xi_0)\in A) \\
	&= \alpha \ell(A) + (1-\alpha) q(x,A) = P (x,A).
	\end{align*}
	
	To sum up, the chains $(X_n)_{n\in\N}$ and $(X_n')_{n\in\N}$ have the same transition kernel, and their initial states also coincide showing that these processes are versions of each other.
\end{proof}

Since we are interested in the distribution of $U_n$, from now on we may and will assume that the
the random walk $(S_n)_{n\in\N}$ is driven by $(X_n')_{n\in\N}$, whereby for every $n\in\N$, each of 
$X_{T_n}$, $S_{T_n}$, $X_{L_n}$, $S_{L_n}$, $L_n$ and $T_n$ is a function of $X_0, S_0$, $(\xi_n)_{n\in\N}$ and $(\eta_n)_{n\in\N}$.

\medskip
In what follows, we are going to prove that the minorization property of $(X_n)_{n\ge 1}$ is 
inherited by $(U_n)_{n\ge 1}$. Let us denote the transition kernel of 
$U$ by $Q: \U\times\B (\U)\to [0,1]$, that is
for all $y\in \U$ and $B\in\B (\U)$
\begin{equation*}
 \P (U_{n+1}\in B\mid U_n=y) = Q(y,B),\quad n>1	
\end{equation*}
holds.
We aim to show that there exist a non-zero Borel measure $\tilde{\ka}:\B (\U)\to [0,\infty)$ such that for all $y\in\U$ and $B\in\B (\U)$,
\begin{equation}\label{eq:UDoeblin}
Q(y,B )\ge \tilde{\ka} (B).
\end{equation}

For $n\in\N_+$, we define $\tau_n = \sup \{t\in\N\mid \eta_{L_n+k}<\alpha,\,k=1,\ldots,t \}$. Clearly,
$\tau_n$ is independent of $U_n$, moreover it follows a $\mathrm{Geo} (1-\alpha)$ distribution counting the number of
failures until the first success i.e. $\P (\tau_n=j)=\alpha^j (1-\alpha)$, $j\in\N$.

Now, let $y=(\ul{x},\ul{s},\ol{x},\ol{s},r)\in\U$ and $B\in\B (\U)$ be arbitrary and fixed. By the tower rule, we have
\begin{align}\label{eq:decomp}
\begin{split}
Q(y,B) &= \P \left(U_{n+1}\in B\mid U_n =y\right)\\ 
&= \sum_{k=0}^{\infty} \P \left(U_{n+1}\in B\mid U_n =y,\,\tau_n=k\right) \P (\tau_n=k\mid U_n =y) \\
&=\sum_{k=0}^{\infty} \P \left(U_{n+1}\in B\mid U_n =y,\,\tau_n=k\right) \alpha^k (1-\alpha)
\\
&\ge
\sum_{k=2}^{\infty} \P \left(U_{n+1}\in B,\,L_{n+1}\le L_n+k\mid S_{L_n} =\ol{s},\,\tau_n=k\right) \alpha^k (1-\alpha),
\end{split}
\end{align}
where we used that the sigma algebras $\sigma (\xi_{L_n+k},\,\eta_{L_n+k},\,k\ge 1)$ and $\sigma (U_n)$ are independent, moreover on sets $\{S_{L_{n}}=\bar{s}, \tau_{n}=k\}$, we have
\begin{equation}
	(X_{L_n+j},S_{L_n+j})= \left(
	h(2\xi_{L_n+j}-1),
	\bar{s}+\sum_{i=1}^{j} h(2\xi_{L_n+j}-1)
	\right),\,\,1\le j\le k
\end{equation}
implying that $U_{n+1}$ and $(X_{T_{n}},S_{T_{n}},X_{L_{n}},L_n-T_n)$ are 
conditionally independent given $\{S_{L_{n}}=\bar{s}, \tau_{n}=k\}$ whenever $k\ge 1$.

Furthermore, we can write
\begin{align}\label{eq:decomp2}
\begin{split}
&\P \left(U_{n+1}\in B,\,L_{n+1}\le L_n+k \mid S_{L_n} =\ol{s},\,\tau_n=k\right)= \\
&=\sum_{l=2}^{k}\P \left(U_{n+1}\in B,\,L_{n+1}= L_n+l\mid S_{L_n} =\ol{s},\,\tau_n=k\right) \\
&=\sum_{l=2}^{k}\sum_{j=1}^{l-1}\P \left(U_{n+1}\in B,\,T_{n+1}= L_n+j,\,L_{n+1}= L_n+l\mid S_{L_n} =\ol{s},\,\tau_n=k\right)
\end{split}
\end{align}

Let us introduce the auxiliary random walk $W_n = \ol{s}+h\sum_{i=1}^{n} (2\xi_i-1)$, and we introduce the associated quantities $L_{0}^W:=0$ and for $n\in\N$,
\begin{equation*}
T_{n+1}^W:=\min\{k>L_{n}^W:\ W_{k}<\ul{\theta}\},\quad
L_{n+1}^W:=\min\{k>T_{n+1}^W:\ W_{k}>\ol{\theta}\}.
\end{equation*}
Similarly, we define $U_n^W=\left(h(2\xi_{T_n^W}-1),W_{T_n^W},h(2\xi_{L_n^W}-1),W_{L_n^W}, L_n^W-T_n^W\right)$, $n\in\N$.

Obviously, for fixed $1\le j<l\le k$, we have
\begin{align}
\begin{split}
	\P &\left(U_{n+1}\in B,\,T_{n+1}= L_n+j,\,L_{n+1}= L_n+l\mid S_{L_n} =\ol{s},\,\tau_n=k\right) =\\
	&=\P (U_1^W\in B,\,T_1^W=j,\,L_1^W=l).
\end{split}
\end{align}

We estimate this probability from below by taking into account only trajectories 
that consist of just one decreasing and one increasing segment (see Figure \ref{fig:path} for
an illustration).
\begin{align}\label{eq:downup}
\begin{split}
& \P(U_1^W\in B,\,T_1^W=j,\,L_1^W=l)\\
&\ge
\frac{1}{2^{l}} 
\P \left(U_1^W\in B,\,T_1^W=j,\,L_1^W=l \middle|
\,\bigcap_{i=1}^{j} \{\xi_i<1/2\},\,\bigcap_{i=j+1}^{l} \{\xi_i\ge 1/2\}\right) \\
&=
\frac{1}{2^{l}} 
\P \left(U_1^W\in B,\,W_{j}<\ul{\theta}\le W_{j-1},\,W_{l-1}\le\ol{\theta}<W_l
\middle|
\,\bigcap_{i=1}^{j} \{\xi_i<1/2\},\,\bigcap_{i=j+1}^{l} \{\xi_i\ge 1/2\}\right)
\end{split}
\end{align}
\begin{figure}[h!]
	\centering
	\includegraphics[width=0.75\linewidth]{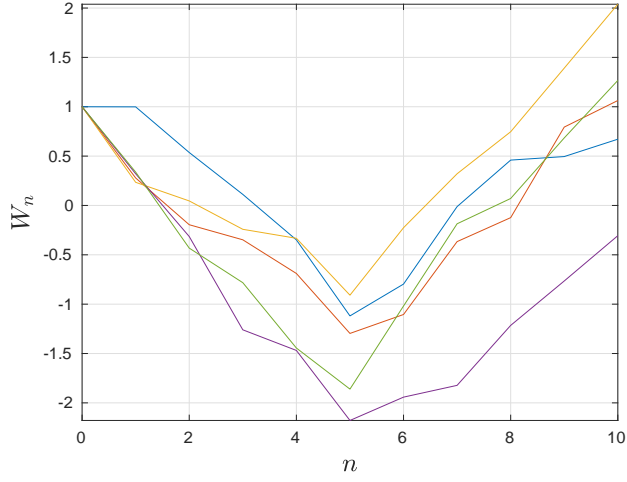}
	\caption{Trajectories of $(W_n)_{n\in\N}$ with one local minimum at $n=5$ ($\ol{s}=1$, $h=1$).}\label{fig:path}
\end{figure}

\noindent
Note that the conditional distribution of $(W_1,\ldots,W_l)$ given $\bigcap_{i=1}^{j} \{\xi_i<1/2\}$ and $\bigcap_{i=j+1}^{l} \{\xi_i\ge 1/2\}$ coincides with the distribution of $(W'_1,\ldots,W'_l)$, where
\begin{equation*}
W'_m:=\ol{s}-h\sum_{i=1}^{\min (j,m)} \xi_i +h\sum_{i=j+1}^{m} \xi_i,\,\, 1\le m\le l
\end{equation*}
with the convention that empty sums are defined to be zero.

Using this, and that $B\subset\U=(-\infty,0)\times (-\infty,\ul{\theta})\times (0,M]\times (\ol{\theta},\ol{\theta}+M)\times (\N\setminus \{0\})$, we can write
\begin{align*}
	\P& \left(U_1^W\in B,\,W_{j}<\ul{\theta}\le W_{j-1},\,W_{l-1}\le\ol{\theta}<W_l
	\middle|
	\,\bigcap_{i=1}^{j} \{\xi_i<1/2\},\,\bigcap_{i=j+1}^{l} \{\xi_i\ge 1/2\}\right) =\\
	&=
	\P\left((-h\xi_j,W'_j,h\xi_l,W'_l,l-j)\in B,\,\,\ul{\theta}\le W'_{j-1},\,W'_{l-1}\le\ol{\theta}\right) \\
	&=\int_0^1\int_0^1\int_0^\infty\int_0^\infty
	\ind_{(-hw,\ol{s}-hu-hw,hz,\ol{s}-hu-hw+hv+hz,l-j)\in B}
	\times\\
	&\times
	\ind_{\ul{\theta}\le\ol{s}-hu}\times
	\ind_{\ol{s}-hu-hw+hv\le\ol{\theta}}\times f_{j-1}(u)\times f_{l-1-j}(v)
	\,\dint u\dint v\dint w\dint z
\end{align*}
where for $m\in\N_+$, $f_m:[0,\infty)\to[0,\infty)$ stands for the probability density function of the sum
of $m$ independent random variables each having a uniform distribution on $[0,1]$.
Now, we can evaluate the quadruple integral using the substitution
$x_1=-hw$, $x_2=\ol{s}-hu-hw$, $x_3=hz$, $x_4=\ol{s}-hu-hw+hv+hz$,
and thus we have
\begin{align*}
	&\int_0^1\int_0^1\int_0^\infty\int_0^\infty
	\ind_{(-hw,\ol{s}-hu-hw,hz,\ol{s}-hu-hw+hv+hz,l-j)\in B}
	\times\\
	&\times
	\ind_{\ul{\theta}\le\ol{s}-hu}\times
	\ind_{\ol{s}-hu-hw+hv\le\ol{\theta}}\times f_{j-1}(u)f_{l-1-j}(v)
	\,\dint u\dint v\dint w\dint z =\\
	&=\frac{1}{h^4}
	\int_{\ol{\theta}}^{\ol{\theta}+M}
	\int_0^h
	\int_{-\infty}^{\ul{\theta}}
	\int_{-h}^0
	\ind_{(x_1,x_2,x_3,x_4,l-j)\in B}
	\times
	\ind_{\ul{\theta}\le x_2-x_1}
	\times
	\ind_{x_4-x_3\le\ol{\theta}}\times \\
	&\times 
	f_{j-1}\left(\frac{\ol{s}-(x_2-x_1)}{h}\right)
	\times
	f_{l-1-j}\left(\frac{x_4-x_3-x_2}{h}\right)
	\,\,\dint x_1\dint x_2\dint x_3\dint x_4 \\
	&=\int_{\mathcal{V}}\ind_{(\mathbf{x},l-j)\in B}\times g_{\ol{s},j,l-j}(\mathbf{x})\,\la(\dint\mathbf{x}),
\end{align*}
where $\mathbf{x}$ is a shorthand notation for $(x_1,x_2,x_3,x_4)$, $\la$ is the 
Lebesgue measure on $\R^4$, and $\mathcal{V}$ is used for 
$(-\infty,0)\times (-\infty,\ul{\theta})\times (0,M]\times (\ol{\theta},\ol{\theta}+M)$, moreover for $j,m>1$ 
\begin{align}\label{eq:g}
\begin{split}
g_{\ol{s},j,m}(\mathbf{x}) &=\frac{1}{h^4}
\ind_{\ul{\theta}\le x_2-x_1}
\times 
\ind_{x_4-x_3\le\ol{\theta}}
\times
\ind_{x_1\in [-h,0]}
\times 
\ind_{x_3\in [0,h]}\\
&\times
f_{j-1}\left(\frac{\ol{s}-(x_2-x_1)}{h}\right) \times f_{m-1}\left(\frac{x_4-x_3-x_2}{h}\right).
\end{split}
\end{align}
Substituting this back into \eqref{eq:decomp2} and reindexing by $m=l-j$ yields

\noindent 
\begin{adjustbox}{max width=\textwidth}
	\parbox{\linewidth}{%
\begin{align*}
	\P \left(U_{n+1}\in B,\,L_{n+1}\le L_n+k\mid S_{L_n}=\ol{s},\,\tau_n=k\right)
	&\ge 
	\sum_{l=2}^{k}
	\sum_{j=2}^{l-2}
	\frac{1}{2^l}
	\int_{\mathcal{V}}\ind_{(\mathbf{x},l-j)\in B}\times g_{\ol{s},j,l-j}(\mathbf{x})\,\la (\dint\mathbf{x})
	\\
	&=
	\int_{\mathcal{V}}
	\sum_{m=2}^{k-2}
	\sum_{j=2}^{k-m}
	\frac{1}{2^{j+m}}
	\ind_{(\mathbf{x},m)\in B}\times g_{\ol{s},j,m}(\mathbf{x})\,\la (\dint\mathbf{x})
\end{align*}
}
\end{adjustbox}

and thus by \eqref{eq:decomp}, we arrive at 
\begin{align*}
	Q(y,B) &\ge \int_{\mathcal{V}}
	\sum_{k=4}^{\infty}
	\alpha^k (1-\alpha)
	\sum_{m=2}^{k-2}
	\sum_{j=2}^{k-m}
	\frac{1}{2^{j+m}}
	\ind_{(\mathbf{x},m)\in B}\times g_{\ol{s},j,m}(\mathbf{x})\,\la (\dint\mathbf{x}) \\
	&= \sum_{m=1}^{\infty}\int_{\mathcal{V}} \ind_{(\mathbf{x},m)\in B}
	\times
	\ind_{m\ge 2}
	\times
	\sum_{k=m+2}^{\infty}
	\sum_{j=2}^{k-m}
	\frac{\alpha^k (1-\alpha)}{2^{j+m}}
	g_{\ol{s},j,m}(\mathbf{x})\,\la (\dint\mathbf{x}) \\
	&= \sum_{m=1}^{\infty}\int_{\mathcal{V}} \ind_{(\mathbf{x},m)\in B}
	\times
	\ind_{m\ge 2}
	\times
	\sum_{j=2}^{\infty}
	\sum_{k=j+m}^{\infty}
	\frac{\alpha^k (1-\alpha)}{2^{j+m}}
	g_{\ol{s},j,m}(\mathbf{x})\,\la (\dint\mathbf{x}) \\
	&= \int_{B} 
	\ind_{m\ge 2}
	\times
	\sum_{j=2}^{\infty}
	\frac{\alpha^{j+m}}{2^{j+m}}
	g_{\ol{s},j,m}(\mathbf{x})\,\la (\dint\mathbf{x})\otimes \delta (\dint m),
\end{align*}
where $\delta$ is the usual counting measure on $\N$.

Let us observe that on $\mathrm{Supp}(g)\subseteq\U$, $\ul{\theta}\le x_2-x_1\le \ul{\theta}+h$.
Now, we fix $0<\tilde{\ga}<\min \left((\ol{\theta}-\ul{\theta})/h,1\right)$ and consider
only $x_1$ and $x_2$ satisfying $\ul{\theta}\le x_2-x_1\le \ul{\theta}+\tilde{\ga}h$. Furthermore, since the jumps of $(S_n)_{n\in\N}$ are bounded from above by $M$, we have $\ol{\theta}<\ol{s}<\ol{\theta}+ M$ hence for the argument of $f_{j-1}$ in \eqref{eq:g}, we get
\begin{equation}
0<
\frac{\ol{\theta}-\ul{\theta}}{h}-\tilde{\ga}
\le
\frac{\ol{s}-(x_2-x_1)}{h}
\le 
\frac{\ol{\theta}-\ul{\theta}}{h}+\frac{M}{h}.
\end{equation}
provided that $\ul{\theta}\le x_2-x_1\le \ul{\theta}+\tilde{\ga}h$.

Introducing $\omega_j=\min\{f_{j-1}(t)\mid (\ol{\theta}-\ul{\theta})/h-\tilde{\ga}\le t\le (\ol{\theta}-\ul{\theta})/h+M/h\}$, we arrive at the estimate 
\begin{equation*}
	f_{j-1}\left(\frac{\ol{s}-(x_2-x_1)}{h}\right) \ge \omega_j
	\times
	\ind_{\ul{\theta}\le x_2-x_1\le \ul{\theta}+\tilde{\ga} h}
\end{equation*}
which is uniform in $\ol{s}$, and $\omega_j>0$ whenever $j>(\ol{\theta}-\ul{\theta})/h+M/h+1$.

If we put all together, for $C_h = \{(\mathbf{x},m)\in\U\mid x_1\in [-h,0];\,\ul{\theta}\le x_2-x_1 \le \ul{\theta}+\tilde{\ga} h;\,x_3\in [0,h];\,x_4-x_3\le\ol{\theta};\,m\ge 2\}$, we obtain
\begin{align*}
Q(y,B)\ge 
	\int_{B\cap C_h}
	\frac{\alpha^m}{h^4 2^m} 
	f_{m-1}\left(
	\frac{x_4-x_3-x_2}{h}
	\right)
	\times
	\sum_{j=2}^\infty\frac{\alpha^j \omega_j}{2^j}
	\,\la (\dint\mathbf{x})\otimes \delta (\dint m),
\end{align*}
where the right hand-side depends only on $\alpha, M, h, \ul{\theta}, \ol{\theta}$, but not on $\ol{s}$ hence \eqref{eq:UDoeblin} holds with
\begin{equation}
	\tilde{\ka} (B)=\int_{B\cap C_h}
	\frac{\alpha^m}{h^4 2^m} 
	f_{m-1}\left(
	\frac{x_4-x_3-x_2}{h}
	\right)
	\times
	\sum_{j=2}^\infty\frac{\alpha^j \omega_j}{2^j}
	\,\la (\dint\mathbf{x})\otimes \delta (\dint m)
\end{equation}
which is obviously a non-zero Borel measure on $\B (\U)$.

To sum up, we showed that the chain $(U_n)_{n\ge 1}$ satisfies the uniform minorization condition \eqref{eq:UDoeblin}, and 
thus by Theorem 16.2.2 in \cite{mt}, there exist a probability measure, independent of $(S_0,X_0)$, such that
$\law{U_n}\to\upsilon$ in total variation. Moreover, by Theorem 17.0.1 in \cite{mt}, for bounded measurable 
functionals of $U_n$ the law of large numbers holds as it is stated in Theorem \ref{cor:LLN}. 
(Actually, even a central limit theorem could be established.) This completes the proof.
\end{proof}

\begin{proof}[Proof of Theorem \ref{thm:ovi}]
	

The idea of the proof is similar to that of Theorem \ref{cor:LLN}, but the details 
are somewhat simpler. We only sketch the main steps. 

We consider the process $Z_n = (X_{L_n}, S_{L_n})$, $n>1$ which is obviously a time-homogeneous Markov chain
on the state space $\triangle:=\{(x,s)\in (0,M]^2\mid x\ge s\}$. In what follows, we prove that chain $(Z_n)_{n>1}$ satisfies a minorization condition similar to \eqref{eq:UDoeblin} in the proof of Theorem
\ref{cor:LLN}. More precisely, we aim to show that there exist a non-zero Borel measure $\beta :\B (\triangle)\to [0,\infty)$ such that for all $z\in\triangle$ and $A\in\B(\triangle)$,
\begin{equation}\label{eq:minZ}
	Q(z,A)\ge\beta (A),
\end{equation}
where $Q:\triangle\times\B (\triangle)\to [0,1]$ is the transition kernel of the chain $(Z_n)_{n\in\N}$.

Let $(\xi_n)_{n\in\N}$, $(\eta_n)_{n\in\N}$, and $(\tau_n)_{n\in\N}$ be as in the the proof of Theorem
\ref{cor:LLN}. For $z=(\ol{x},\ol{s})\in\triangle$ and $A\in\B (\triangle)$ arbitrary and fixed,
by the tower rule, we have
\begin{align}\label{eq:ovi:decomp}
	\begin{split}
		Q(z,A) &= \P \left(Z_{n+1}\in A\mid Z_n =z\right) = \sum_{k=0}^{\infty} \P \left(Z_{n+1}\in A\mid Z_n =z,\,\tau_n=k\right) \alpha^k (1-\alpha)
		 \\
		&\ge 
		\sum_{k=2}^{\infty} \sum_{l=2}^{k}\P \left(Z_{n+1}\in A,\,L_{n+1}= L_n+l\mid S_{L_n} =\ol{s},\,\tau_n=k\right) \alpha^k (1-\alpha),
		\end{split}
	\end{align}
where we applied the same principles as in the derivation of \eqref{eq:decomp}. 
	
Again by introducing the auxiliary random walk	$W_n = \ol{s}+h\sum_{j=1}^{n} (2\xi_j-1)$, and the associated quantities $L_n^W$, $Z_n^W=(h(2\xi_{L_n^W}-1),W_{L_n^W})$, $n\in\N$, where $L_0^W:=0$ and for $n\in\N$,
$L_{n+1}^W:=\inf \{k>L_n^W\mid W_{k-1}\le 0<W_k\}$, for fixed $2\le l\le k$, we can write
\begin{align}\label{eq:alma}
	\begin{split}
		&\P \left(Z_{n+1}\in A,\,L_{n+1}= L_n+l\mid S_{L_n} =\ol{s},\,\tau_n=k\right) = \P (Z_1^W\in A,\,L_1^W=l) \\
		&= \frac{1}{2^l}\P \left((h(2\xi_l-1),W_{l})\in A,\,W_{l-1}\le 0<W_l \middle|
		\,\bigcap_{i=1}^{l-1} \{\xi_i<1/2\},\,\{\xi_l\ge 1/2\}
		\right)
		\end{split}
\end{align}	
where similarly to \eqref{eq:downup}, we have taken into account trajectories decreasing in $l-1>0$ steps
and increasing only in the $l$-th step. For the conditional probability, we have
\begin{align*}
	\P &\left((h(2\xi_l-1),W_{l})\in A,\,W_{l-1}\le 0<W_l \middle|
	\,\bigcap_{i=1}^{l-1} \{\xi_i<1/2\},\,\{\xi_l\ge 1/2\}
	\right)= \\
	&=\int_0^1 \int_0^\infty \ind_{(hv,\ol{s}-hu+hv)\in A}\times 
	\ind_{\ol{s}-hu\le 0}\times f_{l-1}(u)\dint u \dint v \\
	&=\frac{1}{h^2}\int_0^h \int_0^{x_1} \ind_{(x_1,x_2)\in A}
	\times f_{l-1}\left(\frac{\ol{s}-(x_2-x_1)}{h}\right)\dint x_2\dint x_1 \\
	&=\int_{(0,h]^2\cap A}  f_{l-1}\left(\frac{\ol{s}-(x_2-x_1)}{h}\right)\,\la (\dint\mathbf{x})
\end{align*}
where $f_m:[0,\infty)\to[0,\infty)$ is the probability density function of the sum
of $m\ge 1$ independent random variables each having a uniform distribution on $[0,1]$, 
$\mathbf{x}=(x_1,x_2)$, and $\la$ denotes the Lebesgue measure on $\R^2$.
	 
Notice that if $(x_1,x_2)\in(0,h]\cap\triangle$ then $0\le x_1-x_2\le\min (M,h)$, moreover $0<\ol{s}\le M$
hence we have $0\le (\ol{s}-(x_2-x_1))/h\le M/h+\min (M/h,1)$, and thus we obtain
\begin{equation}
	f_{l-1}\left(\frac{\ol{s}-(x_2-x_1)}{h}\right)\ge \omega'_{l}\times \ind_{\tilde{\ga}'\min (M,h)\le x_1-x_2},
\end{equation}
where $\tilde{\ga}'$ can be any fixed number in $(0,1)$, and $\omega_l' = \inf \{f_{l-1}(t)\mid \tilde{\ga}'\min (M/h,1)\le t\le M/h+\min (M/h,1)\}$ that is a positive number not depending on $\ol{s}$ whenever $l>M/h+1$.

If we put all together, we obtain the following lower estimate
\begin{align*}
	Q(z,A)&\ge\int_{(0,h]^2\cap A}  \sum_{k=2}^{\infty}\sum_{l=2}^{k}\frac{\alpha^k (1-\alpha)}{2^l}\omega'_{l}\times \ind_{\tilde{\ga}'\min (M,h)\le x_1-x_2}\,\la (\dint\mathbf{x}) \\
	&=\int_{(0,h]^2\cap A}  
	\ind_{\tilde{\ga}'\min (M,h)\le x_1-x_2} \times
	\sum_{l=2}^{\infty}\frac{\alpha^l \omega'_{l}}{2^l}
	\,\la (\dint\mathbf{x}),
\end{align*}
where the right hand-side depends only on $\alpha,M,h$ and the choice of $\tilde{\ga}'\in (0,1)$, but not depends on $z$, moreover
\begin{equation*}
	\B (\triangle)\ni A\mapsto \beta(A):=\int_{(0,h]^2\cap A}  
	\ind_{\tilde{\ga}'\min (M,h)\le x_1-x_2} \times
	\sum_{l=2}^{\infty}\frac{\alpha^l \omega'_{l}}{2^l}
	\,\la (\dint\mathbf{x})
\end{equation*} 
defines a non-zeros Borel measure on $\B (\triangle)$, and thus \eqref{eq:minZ} holds with this $\beta$.

\smallskip
To sum up, we proved that the chain $(Z_n)_{n\in\N}$ satisfies the uniform minorization condition, and thus
it admits a unique invariant probability measure $\pi$ such that $\law{Z_n}\to \pi$ at a 
geometric rate in total variation as $n\to\infty$ (See for example 
Lemma 18.2.7 and Theorem 18.2.4 in \cite{EricMarkov2018})
which completes the proof of Theorem \ref{thm:ovi}. 

\end{proof}

\begin{remark}\label{inno}{\rm We explain a seemingly innocuous but actually powerful extension of 
some of the arguments above. Let $X_{t}$ be a \emph{time inhomogeneous} Markov chain with kernels
$P_{n}$, $n\geq 1$ such that
$$
P(X_{n+1}\in A|X_{n}=x)=P_{n+1}(x,A),\ x\in\mathcal{X},\ A\in\mathcal{B}(\mathbb{R}),\ n\in\mathbb{N}.
$$
Let Assumption \ref{as:jmpDoeblin} hold for each $P_{n}$, $n\in\mathbb{N}$ (with the same $\alpha,h$){}
and let Assumption \ref{as:zeroMean} hold for each $P_{n}$. 
In this case, $U_{t}$, $t\geq 2$ will be a \emph{time-inhomogeneous} Markov chain and repeating
the argument of the proof for Theorem \ref{cor:LLN} establishes the existence of a probability
$\tilde{\kappa}$ such that $Q_{n}(x,A)\geq \tilde{\kappa}(A)$, $x\in\mathcal{U}$, $A\in\mathcal{B}(\mathcal{U})$, $n\geq 3$,
where $Q_{n+1}(x,A)=P(U_{n+1}\in A|U_{n}=x)$ is the transition kernel of $U$.}
\end{remark}


\begin{remark}
{\rm One could treat certain stochastic volatility-type models where $X_{t}=\sigma_{t}\varepsilon_{t}$
with $\varepsilon_{t}$ i.i.d.\ and $\sigma_{t}$ a Markov process. In this case $X_{t}$ is not Markovian
but the pair $(X_{t},\sigma_{t})$ is. An extension to even more general non-Markovian $X_{t}$ also seems possible.
We do not pursue these generalizations here.}
\end{remark}

\section{Conclusions and future work}\label{sec:con}

It would be desirable to remove the (one-sided) boundedness assumption on the state space 
of $X_{t}$ and relax the minorization condition \eqref{minocond} to some kind of local minorization. Due to the rather
complicated dynamics of $U_{t}$ such extensions do not appear to be straightforward at all. 

Removing the boundedness hypothesis on $u,p$ in Section \ref{sec:trading} would also be desirable but looks challenging. 

Replacing the constant drift $\mu$ by a functional of $X_{t}$ would also significantly extend the 
family of models in consideration.

An adaptive optimization of the thresholds $\underline{\theta},\overline{\theta}$ could be performed
using the Kiefer-Wolfowitz algorithm, as proposed in Section 6 of \cite{kinga}. There are a number of technical
conditions (e.g.\ mixing properties, smoothness of the laws) that need to be checked for applying 
\cite{kinga} but the ergodic
properties established in this article strongly suggest that this programme indeed can be carried out.

Extensions to non-Markovian stochastic volatility models (see \cite{comte-renault,rough}) seem feasible but require
further technicalities.


\end{document}